\documentclass[11pt]{article}
\usepackage{latexsym,url}
\usepackage{amsmath,amssymb,array,graphicx,verbatim}
\usepackage{setspace}
\usepackage{natbib}

\usepackage
  [breaklinks,bookmarks,bookmarksnumbered,bookmarksopen,bookmarksopenlevel=2]
  {hyperref}
{\makeatletter \hypersetup{pdftitle={\@title}}}

{\makeatletter
 \gdef\xxxmark{%
   \expandafter\ifx\csname @mpargs\endcsname\relax 
     \expandafter\ifx\csname @captype\endcsname\relax 
       \marginpar{xxx}
     \else
       xxx 
     \fi
   \else
     xxx 
   \fi}
 \gdef\xxx{\@ifnextchar[\xxx@lab\xxx@nolab}
 \long\gdef\xxx@lab[#1]#2{{\bf [\xxxmark #2 ---{\sc #1}]}}
 \long\gdef\xxx@nolab#1{{\bf [\xxxmark #1]}}
}

\let\realbfseries=\bfseries
\def\bfseries{\realbfseries\boldmath}

\newif\ifabstract
\abstracttrue
\newif\iffull
\ifabstract \fullfalse \else \fulltrue \fi

\let\epsilon=\varepsilon

\newsavebox{\theorembox}
\newsavebox{\lemmabox}
\newsavebox{\remarkbox}
\newsavebox{\corollarybox}
\newsavebox{\propositionbox}
\newsavebox{\examplebox}
\newsavebox{\conjecturebox}
\newsavebox{\algbox}
\newsavebox{\qbox}
\newsavebox{\problembox}
\newsavebox{\definitionbox}
\newsavebox{\assumptionbox}
\newsavebox{\hypothesisbox}
\savebox{\theorembox}{\noindent\bf Theorem}
\savebox{\lemmabox}{\noindent\bf Lemma}
\savebox{\remarkbox}{\noindent\bf Remark}
\savebox{\corollarybox}{\noindent\bf Corollary}
\savebox{\propositionbox}{\noindent\bf Proposition}
\savebox{\examplebox}{\noindent\bf Example}
\savebox{\conjecturebox}{\noindent\bf Conjecture}
\savebox{\algbox}{\noindent\bf Algorithm}
\savebox{\qbox}{\noindent\bf Question}
\savebox{\definitionbox}{\noindent\bf Definition}
\savebox{\problembox}{\noindent\bf Problem}
\savebox{\assumptionbox}{\noindent\bf Assumption}
\savebox{\hypothesisbox}{\noindent\bf Hypothesis}
\newtheorem{theorem}{\usebox{\theorembox}}

\newtheorem{corollary}[theorem]{\usebox{\corollarybox}}

\newcommand{\qed}{\;\;\;\Box}

\newenvironment{proof}{\par{\bf Proof:}}{\(\qed\) \par}
\begin{document}

\title{On Describing the Routing Capacity Regions of Networks}

\author{{Ali Kakhbod}\footnote{Electrical Engineering and Computer Science, University of Michigan, Ann Arbor, MI 48109. Email: {\texttt{akakhbod@umich.edu}}},
 S.~M.~Sadegh~Tabatabaei~Yazdi\footnote{Department of Electrical and Computer Engineering, Texas A\&M University, College Station, TX 77840. Email: \texttt{sadegh@neo.tamu.edu}}}

\maketitle

\begin{abstract}
The routing capacity region of networks with multiple
unicast sessions can be characterized using Farkas' lemma
as an infinite set of linear inequalities. In this paper this result
is sharpened by exploiting properties of the solution satisfied
by each rate-tuple on the boundary of the capacity region,
and a finite description of the routing capacity region which
depends on network parameters is offered. For the special case
of undirected ring networks additional results on the complexity
of the description are provided.

\end{abstract}


\section{Introduction}
Routing protocols  underlie the traditional strategies for
 communicating information in data networks.
The newer paradigm of {\em network coding} (see, e.g., 
\citet{Ahlswede00}) 
offers  potentially more reliable coding schemes with 
higher throughput and error correcting capabilities, 
but it is costlier to
implement (see e.g., \citet{Langberg}).
It is important to better understand routing because of its significance
to the most practical networks. Furthermore, 
routing capacity regions provide inner 
bounds for the corresponding network coding capacity regions, 
and there are cases where the two capacity regions for the 
same networking problem are identical (e.g.,  \citet{KramerSavari06}, \citet{isit07}, \citet{chinese}).

We here focus on the routing capacity regions of a general class of networks supporting multiple multicast sessions. 
Much of the routing literature focuses on the 
{\em multicommodity flow} problem in which every message in the network is 
transmitted from a source to a unique destination (\citet{demos1},\citet{demos2}).
The famous {\em max-flow min-cut} theorem provides bounds on the rates 
of the different messages being simultaneously transmitted between the 
different source-destination pairs.  \citet[Part VII]{Schrijver03} 
surveys many of the cases where this bound is tight.  The paper \citet{hu} 
is an early reference which provides an example where the bound is not tight.

The papers \citet{iri} and \citet{onaga}  establish a result 
sometimes called the ``Japanese theorem'', a special case of Farkas'
lemma, which provides necessary and  sufficient conditions for determining if an arbitrary set of  rates is feasible in a network.
One shortcoming of this result is that the description of the routing
capacity region for the multicommodity flow problem
involves the intersection of an infinite set of inequalities.
Another shortcoming of this result is that it only considers unicast sessions.
While the assumption of a unique destination is natural for many application
areas of network optimization, for communication problems we want to allow
for the possibility of messages from a single transmitter to multiple
receivers. Using standard terminology from communications, we further refer 
to {\em unicast} or {\em multicast} messages to indicate if the set of 
destinations is a single terminal or a set of multiple terminals.  We will 
use the terms unicast and multicommodity flow interchangeably.

Just as one can form a system of linear inequalities to describe 
a multicommodity flow problem, one can likewise study the general
multiple multicast problem where every terminal in the network potentially
has messages for every non-empty subset of the other accessible terminals.
For a multicommodity flow or unicast session the flow for a session 
which enters an intermediate vertex along the path is identical to the flow 
for that session emanating from that vertex.  The natural generalization for 
multicast sessions constrains each spanning subtree carrying flow to have
all of the edges of that subtree transmit the same flow.  The set of flows along the 
various paths and subtrees are jointly constrained by the capacities of the 
edges or nodes in the graph, and the corresponding fractional routing capacity region
can in principle be determined by Fourier-Motzkin elimination 
\citet{Fourier-Motzkin}. However, as the results of Fourier-Motzkin elimination are specific to the set 
of constraints for a particular networking problem, our objective is to offer a
characterization which will apply to many networking problems.

The papers \citet{chinese} and  \citet{isit07} extend the Japanese theorem to the networks with multiple multicast sessions.
Roughly, they characterize the polytope of all achievable rates of a network with capacitated links in terms of  an infinite  set of inequalities, each corresponding to  a distance function that assigns integer distances
to the edges of the network. This characterization is valid for the multicast sessions, where the messages can have more than one destination. We will discuss this result and its proof in Section 2, and describe an
inequality elimination technique to help study the network coding capacity region 
of special cases of the multiple multicast problem on 
an undirected ring network. The technique
determines the minimal necessary and sufficient set of inequalities
among the infinite set of inequalities specified by
the Japanese theorem and is a consequence of properties of
the routing solution for any rate-tuple on the boundary of the
routing capacity region; this technique appears to be new even for the special case of multicommodity flow problems. We
use it to further characterize the minimal set of inequalities
for general directed or undirected networks and for undirected
ring networks.

Our focus in this paper is on the \textit{size} of the coefficients of
the inequalities that appear in the minimal description of the
routing rate region of an undirected network.  We combine the
inequality elimination technique with complexity results (see,
e.g., \citet{Fourier-Motzkin}, \citet{gls}) on the description of a rational system of linear
inequalities to bound the coefficients of the linear inequalities
that describe the routing rate region. The obtained bounds are exponential in the  number of edges of network. We further discuss an
average case analysis of the size of linear inequalities for
undirected ring networks, and by applying a probabilistic technique we suggest that for the characterization of routing capacity region we truly need to take into account the inequalities with the coefficients that grow polynomially in  the  number of edges of network.    

The outline of the paper is as
follows. In Section 2 we formulate  and extend the Japanese theorem and
describe an inequality elimination technique that was recently introduced in \citet{chinese} and \citet{isit07} to help study the
routing capacity region. In Section 3 we present
our results on the complexity of routing capacity regions of
networks.

\section{The Routing Capacity Region of General Networks}
\subsection{Preliminaries}
 Consider a network that is represented by a graph $G(V,E)$, where
 $V$ and $E$ respectively denote the set of vertices and edges in the 
network graph.  The edges are either all undirected, meaning that the sum
of flow along both directions of an edge is bounded by the capacity of the
edge, or all directed.
 Furthermore, for any subgraph $S$ of the network let $V(S)$ and $E(S)$
respectively denote its set of vertices and edges. In a general communication 
setting, every vertex $v \in V$ can simultaneously send messages to arbitrary
 nonempty subsets of accessible vertices in $V \setminus \{v\}$. 
Every message $M$ with source node 
$v_s$ and set of destination nodes $\{ v_1, \ \cdots, \ v_k \}$, is associated
 with a rate $R_M$ and with a set of $t(M)$ spanning subtrees, 
$\{T_M^1, \cdots, T_M^{t(M)} \}$, that connect $v_s$ to $\{v_1, \ \cdots, \ 
v_k\}$.  We assume throughout that rates from any source to a set of vertices
that include an inaccessible destination are always set to zero.

For message $M$, let $r_M^j$ be the amount of flow for that message 
that passes through spanning subtree $T_M^j, \; j \in \{1, \ \dots , \ t(M)\}$.
We then have  
\begin{displaymath}
\sum_{j=1}^{t(M)}r_M^j=R_M.  
\end{displaymath}
Flows of the network satisfy the capacity constraints
\begin{displaymath}
\sum_{M, j : \; e \in E(T_M^j)}{r_M^j} \leq C_e, \; e \in E, 
\end{displaymath}
where  $C_e$ denotes the capacity of edge $e$.

A rate-tuple ${\cal R}=(R_{M_1}, \cdots, R_{M_N})$ corresponding to the
 sessions $M_1, \cdots, M_N$ is said to be {\em feasible} if for each $i \in 
\{1, \ \dots , \ N\}$ there exists an  assignment of $\{r_{M_i}^1, \cdots, 
r_{M_i}^{t(M_i)} \}$ such that  
\begin{displaymath}
\sum_{j=1}^{t(M_i)}r_{M_i}^j=R_{M_i}
\end{displaymath} 
and the  edge constraints are fulfilled. 
Our goal is to offer a new way of thinking about the set of feasible 
rate-tuples in a given network and to provide new characterizations of routing
capacity regions. 
\subsection{Generalizations of the Japanese Theorem}
The Japanese theorem characterizes the set of feasible routing rates-tuples for
edge-constrained networks in problems where there are only multiple unicast 
sessions in terms of an infinite set of inequalities. Each inequality is in terms of a  ``distances function''. A distance function is 
a function that assigns a positive integer to each edge in the network which is called the ``distance'' of the edge.
For each inequality we further need to find the shortest path lengths for each session with respect to the corresponding 
distance function. The length of a path is the sum of distances of the edges on that path and the shortest path
is the path with the shortest length.  
In what follows, $\texttt{Z}^+$ refers to the set of nonnegative integers.
\begin{theorem}[Generalized Japanese theorem for edge-constrained networks \citet{isit07, chinese}]
\label{edge-japanese}
Consider a directed or an undirected edge-constrained network $G(V,E)$.
For function $f: E \rightarrow \texttt{Z}^+$, define
$L_f(T)= \sum_{e \in E(T)}{f(e)}$ and 
$\ell_f(M)= \min_{j \in \{1, \ \dots , \ t(M) \}} {L_f(T_M^j)}$.
The rate-tuple ${\cal R}=(R_{M_1}, \cdots, R_{M_N})$ is feasible in $G(V,E)$ if
and only if for every function $f: E \rightarrow \texttt{Z}^+$,
the following inequality holds:
\begin{equation}
\label{eq:edge-japanese}
\sum_{i=1}^{N}{\ell_f(M_i)R_{M_i}}  \leq \sum_{e \in E}{f(e)C_e}.
\end{equation}
\end{theorem}
To give an intuitive explanation of the Japanese theorem, consider a single unicast session in an edge constrained network. Then the famous max--flow min--cut theorem 
of Ford and Fulkerson provides the capacity of the transmission from the source node to the destination node. It is easy to verify that max--flow min--cut theorem
is equivalent to considering the constraints of the form  (\ref{eq:edge-japanese}) when the distance functions are restricted to the functions with the range $\{0,1\}$ and  each such function corresponds to an edge cut that separates the source node from the destination node. Each such function   assigns value one to all edges that form a cut from the source node to the destination node and all other edges will be assigned a zero value. The Japanaese theorem states that in the case that all sessions are present in the network, the  capacity of the network is characterized by all possible distance functions with the range of positive integers. 
\subsection{An Inequality Elimination Technique}
Theorem \ref{edge-japanese}  is unsatisfying because 
it describes the routing capacity region with infinitely many inequalities, 
and by Fourier-Motzkin elimination we know that the collection of feasible
rate-tuples is a polytope defined by a finite set of inequalities. 
To circumvent this issue, we establish when a Japanese theorem inequality is 
redundant for a networking problem by exploiting the special structure 
of these inequalities.  As we will see later, this approach enables 
us to offer a description of the set of feasible rate-tuples for the
multiple multicast problem with only finitely many inequalities.

Every inequality in \eqref{eq:edge-japanese} is a 
description of a halfspace in the space of rate-tuples $\texttt{R}^N$, and the 
feasible polytope of rate-tuples is the intersection of these half spaces with 
the half spaces corresponding to the non-negativity of rates. 
For the two types of network settings, the boundary points on the feasible set 
of rates respectively belong to the hyperplanes defined by:
\begin{eqnarray}
\label{eq:hyperplane-e}
\sum_{i=1}^{N}{\ell_f(M_i)R_{M_i}}  = \sum_{e \in E}{f(e)C_e},\\
\label{eq:hyperplane-n}
\sum_{i=1}^{N}{\ell_f(M_i)R_{M_i}}  = \sum_{v \in V}{f(v)C_v}.
\end{eqnarray}
The following result provides the necessary and sufficient conditions for a 
{\em feasible} rate-tuple to be on a
hyperplane of the form \eqref{eq:hyperplane-e} or \eqref{eq:hyperplane-n}:
\begin{theorem}
\label{thm:hyperplane} Consider a directed network or an undirected network and fix any distance function $f$.
The feasible rate tuple ${\cal
R}=(R_{M_1},\cdots,R_{M_N})$ is on the hyperplane \eqref{eq:hyperplane-e} or 
\eqref{eq:hyperplane-n} if and only if
there exists a feasible assignment $\{r_{M_i}^j: i \in \{1,\cdots,N\}, j \in \{1, \cdots, t(M_i)\}\}$ 
with the properties:
\begin{enumerate}
\item 
$r_{M_i}^j = 0$ if 
$L_f(T_{M_i}^j) > \ell_f(M_i)$, and for the appropriate setting
\item (Edge-constrained setting) $\sum_{\{i, j : \; e \in E(T_{M_i}^j)\}}{r_{M_i}^j} 
= C_e$ for $f(e) > 0$
\end{enumerate}
\end{theorem}
\begin{proof}
To establish the necessity of Conditions 1 and 2 in 
Theorem~\ref{thm:hyperplane}, suppose that the rate-tuple
${\cal R}=(R_{M_1},\cdots,R_{M_N})$ is feasible and is on the defining
hyperplane \eqref{eq:hyperplane-e} corresponding to function $f$. 
The sum of the flows passing through any edge $e$ in the network at most $C_e$.
By multiplying both sides of this constraint by $f(e)$ and
adding up the resulting inequalities over all edges in the network we find that
\begin{equation}
\label{eq-min}
\sum_{i=1}^{N}\sum_{j=1}^{t(M_i)}{L_{f}(T^j_{M_i})r_{M_i}^j} 
\leq \sum_{e \in E}{f(e)C_e}.
\end{equation}
A lower bound for the left-hand side of the preceding inequality is obtained 
when all sessions are routed along their shortest subtrees with respect 
to $f$:
\begin{equation}
\sum_{i=1}^{N}{\ell_f(M_i)R_{M_i}} \leq 
\sum_{i=1}^{N}\sum_{j=1}^{t(M_i)}{L_{f}(T^j_{M_i})r_{M_i}^j} \leq 
\sum_{e \in E}{f(e)C_e}.
\end{equation}
Since the rate-tuple is on the hyperplane given by 
\eqref{eq:hyperplane-e} by assumption, it follows that Condition 1 holds.
To arrive at a contradiction, suppose next that Condition 2 is invalid.
Hence the rate-tuple ${\cal R}$ is also feasible in
another network with edge capacities $C'_e, \; e \in E$,
in which $C'_e \leq C_e$ for all $e$ with strict inequality for at least one
value of $e$ with $f(e) >0$.
Then Theorem \ref{edge-japanese} implies that
\begin{equation}
\sum_{i=1}^{N}{\ell_f(M_i)R_{M_i}} \leq 
\sum_{e \in E}{f(e)C'_e} <  \sum_{e \in E}{f(e)C_e},
\end{equation}
which contradicts \eqref{eq:hyperplane-e}. Thus Condition 2 holds.

To establish sufficiency, consider a feasible rate-tuple which satisfies
Conditions 1 and 2.  The argument for constraint \eqref{eq-min} applies for
any feasible point, and Condition 2 implies that \eqref{eq-min} can be
replaced by the equality
$\sum_{i=1}^{N}\sum_{j=1}^{t(M_i)}{L_{f}(T^j_{M_i})r_{M_i}^j} 
= \sum_{e \in E}{f(e)C_e}.$
By Condition 1 we further know that
\begin{equation}
\sum_{i=1}^{N}{\ell_f(M_i)R_{M_i}}=
\sum_{i=1}^{N}\sum_{j=1}^{t(M_i)}{L_{f}(T^j_{M_i})r_{M_i}^j} = 
\sum_{e \in E}{f(e)C_e}.
\end{equation}
Hence, the rate-tuple ${\cal R}$ is on the defining hyperplane corresponding to
function $f$, completing the proof.
\end{proof}

We focus here on distance functions $f$ that result in {\em nontrivial}
inequalities with $\ell_f (M_i)>0$ for at least one $i \in \{1, \dots , N\}$.
We say that a nontrivial Japanese theorem inequality is {\em redundant} if for 
any assignment of capacities the {\em feasible} rate-tuples on the 
corresponding defining hyperplane all lie on the hyperplane bounding another 
nontrivial Japanese theorem inequality.
We will next demonstrate that Theorem \ref{thm:hyperplane} implies a way to
establish whether or not an inequality coming from the Japanese theorem 
or its extensions is redundant for a given networking problem.  In the next
section we will show that our inequality elimination technique enables us
to characterize the fractional routing capacity region for the multiple
multicast problem with a finite number of inequalities.
As we will see, the true significance of the distance function is summarized
by the collections of shortest paths for the unicast sessions and
shortest subtrees for the multicast sessions corresponding to that function.
\begin{theorem}[Elimination Theorem]
\label{prop:1}
Given an edge-constrained network with a set of messages 
$\{M_1, \cdots, M_N\}$, consider two nontrivial distance functions 
$f$ and $g$
which are not identical.  The network may be either directed or undirected.  If
\begin{enumerate}
\item for every $e \in E, \; f(e)=0$ whenever $g(e)=0$, and
\item for every session $M_i, i \in \{1, \cdots, N\}$ and for all $j \in \{1, \ \cdots, \ t(M_i)\}$ the
property $L_g(T_{M_i}^j)=\ell_g(M_i)$ implies $L_f(T_{M_i}^j)=\ell_f(M_i)$ 
(but not necessarily the converse),
\end{enumerate}
then the half space \eqref{eq:edge-japanese} corresponding to $g$ is redundant 
in the description of the fractional routing capacity region given the 
half space corresponding to $f$.
\end{theorem}

Before we prove this result, we will illustrate it with an example.
Consider an undirected triangle network with $V= \{1, \ 2, \ 3\}$ 
and suppose $C_{(1,2)} = C_{(2,3)} = C_{(3,1)} = 1$.
We permit all possible unicast and multicast sessions.
Take $g((1,2))=2, \ g((2,3))=1,$ and $g((3,1))=3$.
It is easy to verify
\begin{itemize}
\item $\ell_g(1 \rightarrow 2) = \ell_g(2 \rightarrow 1) = 2$ and
the shortest path is $(1,2)$,
\item $\ell_g(2 \rightarrow 3) = \ell_g(3 \rightarrow 2) = 1$ and
the shortest path is $(2,3)$,
\item $\ell_g(3 \rightarrow 1) = \ell_g(1 \rightarrow 3) = 3$ and
both paths are shortest, and
\item $\ell_g(1 \rightarrow \{2,3\}) = \ell_g(2 \rightarrow \{1,3\}) = 
\ell_g(3 \rightarrow \{1,2\}) = 3$ and the shortest subtree is 
$\{(1,2), (2,3) \}$.
\end{itemize}
Therefore, the half space corresponding to distance function $g$ is
\begin{eqnarray}
\lefteqn{\mbox{\hspace*{-2in}}2(R_{1 \rightarrow 2} + R_{2 \rightarrow 1}) 
+ (R_{2 \rightarrow 3} + R_{3 \rightarrow 2}) 
+3(R_{3 \rightarrow 1} + R_{1 \rightarrow 3}) 
+ 3(R_{1 \rightarrow \{2,3\}} +R_{2 \rightarrow \{1,3\}}
+ R_{3 \rightarrow \{1,2\}})} \nonumber \\ & & \leq
2C_{(1,2)} + C_{(2,3)} + 3 C_{(3,1)} = 6. \label{eq:g}
\end{eqnarray}
Next take $f((1,2))=1, \ f((2,3))=0,$ and $f((3,1))=1$.
Notice that the shortest paths and shortest subtrees for each session
under distance function $g$ remain shortest paths and shortest subtrees 
for the sessions under $f$, although $f$ has a second shortest
path for unicast sessions $1 \rightarrow 2$ and $2 \rightarrow 1$ and
a second shortest subtree for the multicast sessions.
The halfspace corresponding to $f$ is
\begin{equation}
(R_{1 \rightarrow 2} + R_{2 \rightarrow 1}) 
+ (R_{3 \rightarrow 1} + R_{1 \rightarrow 3}) 
+ (R_{1 \rightarrow \{2,3\}} +R_{2 \rightarrow \{1,3\}} 
+ R_{3 \rightarrow \{1,2\}}) \leq C_{(1,2)} + C_{(3,1)} = 2. \label{eq:f}
\end{equation}
The theorem states that (\ref{eq:g}) is redundant for defining the 
routing capacity region in the presence of (\ref{eq:f}).  The reason
is that a polytope is defined by a collection of hyperplanes, and
every {\em feasible} rate-tuple like $R_{1 \rightarrow 2}
= R_{2 \rightarrow 3} = R_{3 \rightarrow 1} = 1, \; R_M=0, \ M \not\in
\{1 \rightarrow 2, \ 2 \rightarrow 3, \ 3 \rightarrow 1 \}$ 
which satisfies (\ref{eq:g}) with equality must also satisfy
(\ref{eq:f}) with equality.  The rate-tuple $R_{1 \rightarrow \{2,3\}} = 2, \;
R_M=0, \ M \not\in \{1 \rightarrow \{2,3\} \}$ is an
example of an {\em infeasible} rate-tuple which satisfies (\ref{eq:g}) with 
equality; it is infeasible because four units of capacity are needed to
transmit two units of multicast traffic, and the network has only three
units of capacity.  For the problem of characterizing the routing capacity 
region of a network we can ignore the infeasible rate-tuples.


We assert here that there are 
generalizations of Theorem \ref{prop:1} to some other classes of systems
of linear inequalities.  To save space we only give the proof for the 
edge-constrained setting from \citet{isit07}. \citet{chinese} offers an 
alternate proof of the same result which uses the formalism of linear algebra.
\begin{proof}
Consider a feasible rate-tuple on the hyperplane \eqref{eq:edge-japanese} 
corresponding to $g$.
By Condition 1 of Theorem \ref{thm:hyperplane}, every session is routed only 
along the shortest paths and shortest subtrees associated with $g$, and hence
by assumption only along the shortest paths and shortest subtrees 
corresponding to $f$. Furthermore, note that by assumption any edge $e$ with 
$f(e) > 0$ implies $g(e)>0$
and so Condition 2 of Theorem \ref{thm:hyperplane} implies that this edge must 
be fully utilized. By Theorem \ref{thm:hyperplane} it follows that the 
feasible rate-tuple is also on the hyperplane \eqref{eq:edge-japanese} 
corresponding to $f$.  Since the routing capacity region can be described 
in terms of its defining hyperplanes, the bound corresponding to $g$ is 
redundant given the inequality corresponding to $f$.
\end{proof}

The paper \citet{isit07} considers undirected, edge-constrained ring networks 
with multiple unicast and {\em broadcast} sessions; in a broadcast session
the source transmits a message to all of the other vertices in the network.
That paper introduced the inequality elimination technique and used it to 
prove that distance functions with range $\{0,1\}$ are sufficient for 
characterizing the capacity region. 
The paper \citet{chinese} proved that $\{0,1\}$ edge distances are also 
enough to describe the multiple multicast capacity region of undirected, 
edge-constrained ring networks in which the source and destination vertices of
each communication session form a string of adjacent vertices.
We next present new consequences of the inequality elimination theorem.

\section{On the Complexity of the Routing Capacity Region}

We next consider the complexity of the routing capacity
region for an undirected graph. Let $p$ and $q$ be relatively prime
integers and let $\alpha = p/q$. Define 
$$
size(\alpha)=1+\left\lceil \log_2(1+|p|)\right\rceil +\left\lceil \log_2(1+|p|)\right\rceil.
$$

For the rational vector $\textbf{c}=(\gamma_1,\cdots,\gamma_n)$ and the rational matrix $A=(\alpha_{i,j})_{1\leq i \leq m, 1 \leq j \leq n}$ we have: 
\begin{eqnarray}
size(\textbf{c})=n+size(\gamma_1)+\cdots+size(\gamma_n) \\
size(A)=mn+\sum_{m,n}size(\alpha_{i,j})
\end{eqnarray}

Let $\textbf{x}=(x_1,x_2,\cdots,x_n)^T$. Then the size of the linear inequality $\textbf{a}\textbf{x}\leq \alpha$ is defined as $1+size(\textbf{a})+size(\alpha)$. The size of a system $A\textbf{x}\leq \textbf{b}$ of linear inequalities is defined as $1+size(A)+size(\textbf{b})$. Next let $P\subset \mathcal{R}^n$ be a rational polyhedron. The \textit{facet complexity} of $P$ defined as the smallest number $\Lambda \geq n$ for which there exists a system $A\textbf{x}\leq\textbf{b}$ of rational linear inequalities defining $P$ and each inequality in $A\textbf{x}\leq \textbf{b}$ has size at most $\Lambda$.\\

Consider an undirected network $G(V,E)$ and the rate-tuple
$\mathcal{R} = (R_{M_1} ,\cdots ,R_{M_N})$. Let $P$ denote the set of achievable
rate-tuples in $\mathbb{R}^N$. Theorem \ref{prop:1} provides a systematic method
to characterize the minimal description of $P$ for the general
multiple multicast problem. Here we wish to establish upper
and lower bounds on the maximum values of the functions that
appear in the minimal description of $P$. 
\\

In the case of an undirected ring network with $|E|$ edges, \citet{isit07} 
points out that for the general multiple multicast problem there is a weight 
vector with maximum edge distance $\lfloor (|E|-1)/2 \rfloor$ that cannot be
eliminated.\\

We next show that for undirected ring networks we have to 
consider some distance functions where the maximum edge distance grows
at least exponentially in $|E|$. In words, we derive a lower bound  on the maximum values of the distance functions needed for the minimum description of the corresponding routing rate region, which is exponential in $|E|$. For that matter, we construct a distance function $g$  that can not be eliminated  by any nontrivial distance function with the maximum edge distance less than  $2^{\lfloor (|E|-2)/3 \rfloor}$. The distance function $g$ is not unique and it is constructed as to achieve a  bound  in the same order of the upper bound  obtained in  Theorem 6.
\\ 

\begin{theorem}
\label{th3}
Let $G(V,E)$ be an undirected ring network with vertices $1, \ 2, \ \cdots, \ 
|E| $ in a clockwise direction. For $i \in \{1, \ 2, \ \dots , \ |E|-1 \}$, let
edge $i$ connect vertices $i$ and $i+1$, and let edge $|E|$ connect vertices
$|E|$ and $1$.  There exist a distance function $g$ that cannot be
eliminated by any nontrivial distance function $f$ with 
$\max_{e \in E} f(e) < 2^{\lfloor (|E|-2)/3 \rfloor}$.
\end{theorem}
\begin{proof}
Consider a multicast session with $k-1$ destinations, and suppose the source 
and destination vertices form the set $\{v_1, \ v_2 , \ \dots , \ v_k \}$,
where $1 \leq v_1 < v_2 < \dots < v_k \leq |E|$.
Observe that a minimal spanning subtree is the subgraph consisting of the 
original network except for the vertices $v_j +1, \dots , v_{j+1}-1$ and 
edges $v_j , \dots, v_{j+1}-1$ for some $j \in \{1, \dots , k \}$ 
(with $v_{k+1}=v_1)$.  Therefore, for any distance function the shortest
paths or shortest subtrees for this collection of sessions will correspond
to the longest paths $v_j , \dots, v_{j+1}-1, \; j \in \{1, \dots , k \}$.

Let $\beta = 2^{\lfloor (|E|-2)/3 \rfloor}$.  
Suppose we consider the distance function
\begin{displaymath}
g(e) = \left\{ \begin{array}{ll} \beta, & e \equiv 1 \! \pmod 3 \\
2^{\lfloor ({e-2})/{3} \rfloor}, & \mbox{otherwise} \end{array} \right.
\end{displaymath}
and try to find another distance function $f$ that eliminates $g$.
Since the shortest broadcast trees are preserved under $f$, it follows
that 
\begin{equation}
\max_{i \in E} f(i) = f(e), \; e \equiv 1 \! \pmod 3 . \label{eq:r1}
\end{equation}
Furthermore, for $s \in \{ 2, \ \dots , \ \lfloor |E| / 3 \rfloor \}$
consider the multicast session consisting of all vertices except $3s-4, \
3s-3,$ and $3s-1$.  Under $g$, the path consisting of edges $3s-5, \ 3s-4, \
3s-3$, and the path consisting of edges $3s-2$ and $3s-1$ are both longest,
and therefore (\ref{eq:r1}) implies
\begin{equation}
f(3s-4) + f(3s-3) = f(3s-1), \; s \in \{2, \ \dots , \ \lfloor |E|/3 \rfloor\}.
\label{eq:r2}
\end{equation}
Finally, for $s \in \{ 2, \ \dots , \ \lfloor (|E|+1) / 3 \rfloor \}$
consider the multicast session consisting of all vertices except $3s-4$ and
$3s-2$. Under $g$, the path consisting of edges $3s-5$ and $3s-4$ 
and the path consisting of edges $3s-3$ and $3s-2$ are both longest,
and therefore (\ref{eq:r1}) implies
\begin{equation}
f(3s-4) = f(3s-3), \; s \in \{2, \ \dots , \ \lfloor (|E|+1)/3 \rfloor\}.
\label{eq:r3}
\end{equation}
By (\ref{eq:r2}) and (\ref{eq:r3}), we see that
\begin{equation}
2f(3s-4) = f(3s-1), \; s \in \{2, \ \dots , \ \lfloor |E|/3 \rfloor\}.
\label{eq:r4}
\end{equation}
Equations (\ref{eq:r1})-(\ref{eq:r4}) imply that $f(e)=f(2) \cdot g(e)$ for
all $e \in E$.
\end{proof}

Let $\Lambda^*$ denote the maximum distance among distance
functions used for a shortest description of $P$. Theorem \ref{th3}
establishes that $\Lambda^*\geq 2^{\left\lfloor (|E|-2)/3\right\rfloor}$. We next extend Theorem
\ref{th3} to any undirected graph.

\begin{corollary}
Given undirected graph $G(V,E)$ with maximum
cycle length $m$, for the networking problem in which all possible
multicast sessions are supported, the minimum description
of the corresponding routing rate region requires a distance
function with  $\Lambda^*\geq 2^{\left\lfloor (|E|-2)/3\right\rfloor}$.
\end{corollary} 

\begin{proof}
Let $\mathcal{C}$ denote a maximum cycle of $G$. We extend the proof of Theorem \ref{th3} by using the distance function $f$ along $\mathcal{C}$ and setting $f(e)$, $e \notin \mathcal{C}$, to be sufficiently large. 
\end{proof}

Next we bound $\Lambda^*$ from above.

\begin{theorem}
\label{th4}
For an undirected or a directed network $G(V,E)$, $\Lambda^*\leq 2^{24|E|^3+8|E|^2}$.
\end{theorem}

\begin{proof}
Suppose that the distance vector
$\textbf{f} = (f(1),\cdots , f(|E|))$ belongs to the minimal description
of $P$. We form the homogeneous set of inequalities $A\textbf{g} \leq0$
such that ${\textbf{g} : A\textbf{g} \leq 0, \textbf{g} \in \mathbb{Z}^{|E|}}$ is the set of all
distance vectors that can eliminate $\textbf{f}$ by the criteria given
in Theorem \ref{prop:1}. This includes all inequalities that describe
the shortest subtrees for every session corresponding to
function $f$, and also the non-negativity of elements of $\textbf{g}$.
Notice that this set is non-empty since $\textbf{f}$ is a solution to it.
Furthermore, all elements of matrix $A$ are in $\{0, +1,-1\}$.
This implies the upper bound $3|E| + 1$ on the size of the
inequalities in $A\textbf{g} \leq 0$. Therefore the facet complexity of
$A\textbf{g} \leq 0$, is at most $\Lambda_A = 3|E| + 1$. [\citet{Fourier-Motzkin}, Theorem 10.2]
implies that $A\textbf{g} \leq 0$ has a rational solution of size at most
$4|E|^2\Lambda_A = 12|E|^3 + 4|E|^2$. Let $\textbf{g}_r = (p_1/q_1,\cdots , p_{|E|}/q_{|E|})$
denote such a solution. Since $A\textbf{g} \leq 0$ is a homogeneous set of
inequalities, any integral multiple of $\textbf{g}_r$ is also a solution to
$A\textbf{g} \leq 0$. Now consider the vector $\textbf{g}_z = (q_1\cdots q_{|E|})\textbf{g}_r$.
Clearly $\textbf{g}_z \in \{\textbf{g} : A\textbf{g} \leq 0, \textbf{g}\in \mathbb{Z}^{|E|}\}$, so it can
eliminate $\textbf{f}$. Let $\textbf{g}_z(i)$ be the maximum entry of $\textbf{g}_z$.
Then  $size(\textbf{g}_z(i)) \leq size(q_1 \cdots q_{|E|}) + size(\textbf{g}_r(i))$. Since
$size(q_1 \cdots q_{|E|}) \leq size(\textbf{g}_r)$ and $size(\textbf{g}_r(i)) \leq size(\textbf{g}_r)$,
then $size(\textbf{g}_z(i)) \leq 24|E|^3 + 8|E|^2$. This yields the result.
\end{proof}

The following result, suggests that a small
fraction of the distance functions in our characterization of the fractional
routing capacity region are truly needed and that most distance functions
can be eliminated by distance functions where the maximum entry grows
polynomially with $|E|$.\\

\begin{theorem}
Let $G(V,E)$ be an undirected ring network with edges labeled 
$1, \ 2, \ \cdots, \ |E|$ in a clockwise order. 
Choose any integer $m \geq 6$, and suppose 
$g_{\max} = \max_{e \in E} g(e) > g^{*} \doteq 
{|E|^m}/({1 - \frac{|E|^m}{g_{\max}}})$.
Assume without loss of generality that $g(|E|) = g_{\max}$ 
and for $e \in E \setminus |E|$ let $g(e)$ be uniformly chosen among
nonnegative integers less than or equal to $g_{\max}$.  Then with probability
at least $1- \frac {4}{|E|^{m-5}}- \frac{1}{|E|^{m(|E|-1)}}$ we can find a distance function $f$ with 
$f_{\max} \leq g^{*}$ that eliminates $g$.
\end{theorem}
\begin{proof}
Given distance function $g$ with $g_{\max} \geq g^*$, let
$\phi = \lfloor g_{\max} / |E|^m \rfloor$, and define 
\begin{displaymath}
f(e)=g(e)-(g(e) \pmod \phi), \; e \in E.
\end{displaymath}
Distance function $f$ eliminates distance function $g$ if for any pair of
edge-disjoint subtrees $E_1$ and $E_2$ of $E$, the condition
$\sum_{e \in E_1}{g(e)} \leq \sum_{i \in E_2}{g(e)}$ implies
$\sum_{e \in E_1}{f(e)} \leq \sum_{i \in E_2}{f(e)}$.  Let 
${\cal E}_{E_1,E_2}$ be the event that 
$\sum_{e \in E_1}{g(e)} \leq \sum_{e \in E_2}{g(e)}$ and
$\sum_{e \in E_1}{f(e)} > \sum_{e \in E_2}{f(e)}$. Define
\begin{eqnarray*}
\Delta_g & = & \sum_{e \in E_1}{g(e)} - \sum_{e \in E_2}{g(e)} \\
\mbox{ and } \; \Delta_f & = & \sum_{e \in E_1}{f(e)} - 
\sum_{i \in E_2}{f(e)}.  
\end{eqnarray*}
Since $0 \leq g(e)-f(e) < \phi$ for all $e \in E$, it follows that 
\begin{displaymath}
|\Delta_g-\Delta_f| \; \leq \; \sum_{e \in E}{|g(e)-f(e)|} \; < \; \phi \cdot
|E|. 
\end{displaymath}
We know that $\Delta_g \leq 0$ and $\Delta_f > 0$, and so 
$|\Delta_g| < \phi \cdot |E|$. Let 
$E_{\min} = \min_{e \in E_1 \cup E_2} e$.
Observe that $E_{\min} \neq |E|.$ 
Given $g(e), \leq e \in E \setminus E_{\min}$, there are at most 
$2\phi \cdot |E|$ choices for $g(E_{\min})$ that result in $- \phi \cdot |E|
< \Delta_g \leq 0$. Furthermore $g(E_{\min})$ is a random variable uniformly
distributed over the integers between $0$ and $g_{\max}$. 
Therefore,
\begin{displaymath}
{\mathbb P}({\cal E}_{E_1,E_2}) \; \leq \; \frac{2\phi \cdot |E|}{g_{\max}+1} 
\; < \; \frac{2 \cdot \frac{g_{\max}}{|E|^m} \cdot |E|}{g_{\max}}
\; = \; \frac{2}{|E|^{m-1}}. 
\end{displaymath}
The number of pairs of edge-disjoint subtrees $E_1$ and $E_2$ we need consider
is less than $2|E|^4$. Hence,
\begin{displaymath}
{\mathbb P} \left(\bigcup_{E_1,E_2} {\cal E}_{E_1,E_2} \right) \; < \;
2|E|^4 \cdot \frac{2}{|E|^{m-1}} \; = \; \frac{4}{|E|^{m-5}}.
\end{displaymath}
In addition, In order to omit the trivial cases, a distance function is not trivial if and only if there exist at least $e$ and $e'$ in $E$ such that  $f(e)>0$  and $f(e')>0$ , since for all $e \in E$, ${\mathbb P} (f(e)=0)<\frac{1}{|E|^{m}}$, Thus, with probability
at least $1- \frac {4}{|E|^{m-5}}- \frac{1}{|E|^{m(|E|-1)}}$ we can use distance function $f$ to eliminate
$g$.  Since $f(e) \pmod \phi =0$ for all $e \in E$, we can eliminate 
distance function $f$ by distance function $f^{*}$ with
$f^*(e) = f(e) / \phi, \; e \in E$. Notice that for all $e\in E$,
\begin{eqnarray}
& f^*(e) \; = \; \frac{f(e)}{\phi} \; \leq \; \frac{g(e)}{\phi} 
\; \leq \; \frac{g_{\max}}{\lfloor g_{\max} / |E|^m \rfloor} 
\; < \; \frac{g_{\max}}{\frac{g_{\max}}{|E|^m} -1} \nonumber \\
& \; = \; \frac{|E|^m}{1-\frac{|E|^m}{g_{\max}}} .
\end{eqnarray}
\end{proof}
\section*{Acknowledgements}
The authors are  gratefully thankful to  Serap Savari for her insightful and encouraging discussions on the material of this paper. 
\bibliographystyle{plainnat}
\bibliography{mmor_refrence}

\end{document}